\documentclass[12pt,letterpaper]{amsart}
    \usepackage{braket}

\usepackage[hmargin=1in,vmargin=1in]{geometry}
\usepackage{amsthm}
\usepackage{tabularx}
\usepackage{tikz}
\usetikzlibrary{tikzmark,calc}
\usepackage{parskip}
\usepackage{shuffle}

\usepackage[english]{babel} 
\usepackage[latin1]{inputenc}
\usepackage{amsmath, bm} 
\usepackage{amsfonts}
\usepackage{amssymb}
\usepackage{pifont}
\usepackage{stmaryrd}
\usepackage{latexsym} 
\usepackage{graphicx}
\usepackage{subfigure}
\usepackage[hyphens]{url}
\usepackage{hyperref}
\usepackage{verbatim}
\usepackage[all]{xy}
\usepackage{graphics}
\usepackage{pdfsync}
\usepackage{xcolor}

\usepackage{listings}
\usepackage[noend]{algpseudocode}

\usepackage{biblatex}
\usepackage{csquotes}
\usepackage{ytableau}
\usepackage{textcmds}
\usepackage{latexsym}

\newtheorem{theorem}{Theorem}[section]
\newtheorem{lemma}[theorem]{Lemma}
\newtheorem{corollary}[theorem]{Corollary}

\theoremstyle{definition}
\newtheorem{definition}[theorem]{Definition}
\newtheorem{example}[theorem]{Example}

\newtheorem{remark}[theorem]{Remark}
\newtheorem*{acknowledgement}{Acknowledgement}
\newcommand*\circled[1]{\tikz[baseline=(char.base)]{
            \node[shape=circle,draw,inner sep=2pt] (char) {{#1}};}}
            
\newlength{\cellsize}
\newcommand\tableau[1]{
\vcenter{
\let\\=\cr
\baselineskip=-16000pt
\lineskiplimit=16000pt
\lineskip=0pt
\halign{&\tableaucell{##}\cr#1\crcr}}}

\newcommand{\tableaucell}[1]{{%
\def \arg{#1}\def \void{}%
\ifx \void \arg
\vbox to \cellsize{\vfil \hrule width \cellsize height 0pt}%
\else
\unitlength=\cellsize
\begin{picture}(1,1)
\put(0,0){\makebox(1,1)[c]{$#1$}}
\put(0,0){\line(1,0){1}}
\put(0,1){\line(1,0){1}}
\put(0,0){\line(0,1){1}}
\put(1,0){\line(0,1){1}}
\end{picture}%
\fi}}
\newcommand{\qs}{\tilde{s}}
\newcommand{\qpartial}{\tilde{\partial}}
\newcommand{\qpi}{\tilde{\pi}}
\newcommand{\qk}{\tilde{\kappa}}
\newcommand{\qbarpi}{\tilde{\theta}}
\newcommand{\key}{\kappa}
\newcommand{\qatom}{\tilde{\mathcal{A}}}
\newcommand{\wdash}{\vDash_w}

\DeclareMathOperator{\csum}{cumsum}
\DeclareMathOperator{\strong}{flat}
\DeclareMathOperator{\sort}{sort}

\DeclareMathOperator{\wt}{wt}

\DeclareMathOperator{\qshift}{fshift}

\newcommand{\qslide}{\mathfrak{F}}

\DeclareMathOperator{\SSet}{SS}
\DeclareMathOperator{\qSSet}{FSS}
\DeclareMathOperator{\Des}{Des}

\newcommand{\doml}{\trianglelefteq}
\newcommand{\domr}{\trianglerighteq}

\newcommand{\atom}{\mathcal{A}}
\newcommand{\HSSF}{\text{HSSF}}

\newcommand{\qamp}[2]{\text{QAMP}_{#1,#2}}

\newcommand{\sgn}[1]{\operatorname{sign}(#1)}
\newcommand{\sym}[1]{\text{Sym}_{#1}}
\newcommand{\qsym}[1]{\text{QSym}_{#1}}
\newcommand{\nonsym}[1]{\mathbb{C}[x_1,x_2,\cdots,x_{#1}]}
\newcommand{\Sn}[1]{S_{#1}}
\newcommand{\fparta}{\mathcal{L}_\alpha}

\title{Quasisymmetric divided difference operators and polynomial bases}
\author{Angela Hicks and Elizabeth Niese }
\date{\today}

\begin{document}

\begin{abstract}
    The key polynomials, the Demazure atoms, the Schubert polynomials, and even the Schur functions can be defined using divided difference operator.  In 2000, Hivert introduced a quasisymmetric analog of the divided difference operator.  In particular, replacing it in a natural way in the definition of the Schur functions gives Gessel's fundamental basis.  This paper is our attempt to apply the same methods to define the remaining bases and study the results.  In particular, we show both the key polynomials and Demazure atoms have natural analogs using Hivert's operator and that the resulting bases occur independently and defined by other means in the work of Assaf and Searles, as the fundemental slide polynomials and the fundamental particle basis respectively.  We further explore properties of these two bases, including giving the structure constants for the fundamental particle basis.
\end{abstract}
\maketitle
Several well studied families of polynomial bases for $\mathbb{C}[x_1,x_2,\cdots,x_n]$ are naturally defined using divided difference operators  $$\partial_i:=\frac{1-s_i}{x_i-x_{i-1}}$$ including the key polynomials, the Demazure atoms, and the Schubert polynomials.  Product rules for these polynomials are of particular interest. A classical result in algebraic geometry gives that the product of Schubert polynomials is a positive sum of Schubert polynomials, but despite a combinatorial model for the polynomials, it is a well studied and unsolved problem to give a combinatorial proof of this fact or a combinatorial formula for the structure constants. While the product of key polynomials is not positive in the key polynomials, it is conjectured that the product of key polynomials is atom positive and the product of atoms is atom positive in many known special cases.

These families of polynomials are closely related to other well studied bases: the Demazure atoms and the key polynomials can both be found from a well studied basis for $\mathbb{C}(q,t)[x_1,x_2,\cdots,x_n]$: the nonsymmetric Macdonald polynomials.  By specializing $q$ and $t$ in the nonsymmetric Macdonald polynomials, it was shown by Mason in \cite{mason2009explicit} and Ion in \cite{ion2003nonsymmetric} (respectively) that one recovers the Demazure atoms and key polynomials.  This observation leads to combinatorial models for each of the bases. 

Moreover, divided difference operators can be used to create bases for other vector spaces: for example, a basis for the symmetric functions on $n$ variables, the Schur polynomials, is a special case of both a Schubert polynomial and a key polynomial. Thus Schur polynomials can also be defined using the same divided difference operators. Here, of course, products are well understood, and the product of two Schur functions is a sum of Schur functions, with the well known Littlewood-Richardson rule giving the coefficients.

Hivert~\cite{Hiv2000} gave a quasisymmetric analog of the divided difference operator, $\qpartial_i$ and showed that replacing $\partial_i$ with $\qpartial_i$ in a natural way in the divided difference definition of the Schur functions generated Gessel's fundamental basis for the quasisymmetric functions ($\qsym{n}$).  

In this work, our goal is to study bases resulting from replacing $\partial_i$ with $\qpartial_i$ in a natural way in the divided difference definitions of the key polynomials, the Demazure atoms, and the Schubert polynomials.  In two of the three cases (the key polynomials and the Demazure atoms), the result is a basis for $\nonsym{n}$.  After defining the two, we were surprised to find they appear elsewhere in the literature as the fundamental slide polynomial in Assaf and Searles~\cite{AS2017} and the fundamental particle basis in Searles~\cite{Sea2020} respectively, defined by other means and without a clear connection in the literature to their classical analogs. In the third case (the Schubert polynomials) for reasons we will briefly explain, there appears to be no direct analogue.  The fundamental slide polynomials in particular have found applications in the study of several other important families of polynomials, including the Grothendieck polynomials in \cite{MR2377012}, the nonsymmetric Macdonald polynomials in \cite{MR3864395}, and the Schubert polynomials in \cite{assaf2017combinatorial}. 

By connecting the work of Hivert to the work of Assaf and Searles, we're able to show  their fundamental slide polynomials have a representation theoretic interpretation analogous to that of the key polynomials (otherwise known as Demazure characters in this context).  We also are able to give a number of new product formulas, including giving explicit structure constants for the fundamental particle basis and a refinement of the Littlewood-Richardson formula for the fundamental particle basis given by Searles in \cite{Sea2020}.
\section{Background}

There are three vector spaces on $n$ variables relevant to this work: the symmetric functions ($\sym{n}$), the quasisymmetric functions($\qsym{n}$), and the ``nonsymmetric'' polynomials ($\nonsym{n}$).  
As vector spaces, we have $\sym{n}\subset\qsym{n}\subset \nonsym{n}$.  Their dimensions are equinumerous with partitions of $n$, strong compositions of $n$, and weak compositions of $n$ respectively.

\begin{definition}[Composition]
   Let $\alpha=(\alpha_1,\cdots,\alpha_k)$ be a sequence of nonnegative integers that sum to $n$. Then $\alpha$ is a \textbf{weak composition} of $n$ and we write $\alpha\models_w n$.  If the elements are strictly positive,  $\alpha$ is a \textbf{(strong) composition} of $n$.  Occasionally it is useful to add zeroes to the end of such a composition to make the sequence longer, usually in the context of the notation $$x^\alpha=x_1^{\alpha_1}x_2^{\alpha_2}\cdots x_n^{\alpha_n}.$$  In either case, we say the length of $\alpha$ ($\ell(\alpha)$) is $k$.  
   If the entries of a strong composition are weakly decreasing, then $\alpha$ is a partition $\alpha \vdash n$.
\end{definition}
Usually we reserve $\lambda$ for partitions, and $\alpha$ and $\beta$ for compositions or weak compositions.

\begin{definition}[Symmetric]
A polynomial $p(x_1,\cdots x_n)$ in $\nonsym{n}$ is \textbf{symmetric} ($p(x_1,\cdots,x_n)\in \sym{n}$) if for every $\alpha=(\alpha_1,\alpha_2,\cdots,\alpha_n)$ with nonnegative integer entries and every $\sigma \in \Sn{n}$,  the coefficient of $x_{1}^{\alpha_1}\cdots x_{n}^{\alpha_n}$ in $p(x_1,\cdots x_n)$ is the same as the coefficient of  $x_{\sigma_1}^{\alpha_1}\cdots x_{\sigma_n}^{\alpha_n}$. i.e. 
 $$\left.p(x_1,\cdots,x_n)\right|_{x_{1}^{\alpha_1}\cdots x_{n}^{\alpha_n}}=\left.p(x_1,\cdots,x_n)\right|_{x_{\sigma_1}^{\alpha_1}\cdots x_{\sigma_n}^{\alpha_n}}$$
\end{definition}
Bases for the symmetric functions are usually indexed by partitions of $n$ ($\lambda\vdash n$), as for example the Schur functions $\{s_\lambda(x_1,\cdots,x_n)\}_{\lambda \vdash n}$.   An excellent introduction to symmetric functions for the unfamiliar reader is \cite{stanley2023enumerative}.
\begin{definition}[Quasisymmetric]
    A polynomial $p(x_1,\cdots x_n)$ in $\nonsym{n}$ is \textbf{quasisymmetric} ($p(x_1,\cdots,x_n)\in \qsym{n}$) if for every  composition $\alpha\models n$  $$\left. p(x_1,\cdots x_n)\right|_{x_{i_1}^{\alpha_1}\cdots x_{i_k}^{\alpha_k}}=\left. p(x_1,\cdots x_n)\right|_{x_{j_1}^{\alpha_1}\cdots x_{j_k}^{\alpha_k}}$$ 
    for any $1\leq i_1<i_2\cdots< i_k\leq n\text{ and }1\leq j_1<j_2\cdots< j_k\leq n.$
\end{definition}
We thus say these polynomials are ``shift invariant."

Bases for the quasisymmetric functions are usually indexed by strong compositions of $n$ ($\alpha\models n$) or the equinumerous subsets of $[n-1]=\{1,2\cdots,n-1\}$,
as for example Gessel's fundamental basis $\{F_\alpha(x_1,\cdots, x_n)\}_{\alpha \models n}$.    
\begin{definition}[Gessel's fundamental basis]  Let $$\operatorname{set}(\alpha)=\{\alpha_1,\alpha_1+\alpha_2,\cdots, \alpha_1+\alpha_2+\cdots +\alpha_{k-1}\}\subset[n-1].$$  Then Gessel's fundamental basis for the quasisymmetric functions is defined by $$F_\alpha(x_1,\cdots, x_n) =\sum\limits_{\substack{1\leq i_1\leq i_2\leq\cdots\leq i_k \\ j\in \operatorname{set}(\alpha)\Rightarrow i_j<i_{j+1}}} x_{i_1}x_{i_2}\cdots x_{i_k}.$$
\end{definition}
The quasisymmetric functions show up in a number of well studied conjectures about the representation theory of the symmetric group and have independent significance as they encode the representation theory of the 0-Hecke Algebra.  An excellent introduction to quasisymmetric functions for the unfamiliar reader is \cite{luoto2013introduction}.
Finally, bases for $\nonsym{n}$ are generally indexed by either weak compositions of $n$ or permutations of $n$, which are equinumerous.

\subsection{Bases from divided difference operators}

We begin by defining the classical key polynomials, Demazure atoms, and Schubert polynomials using divided difference operators.

Let $S_n$ act on the polynomials in the usual way by permuting indices.  Using $s_i$ for the simple transposition which acts by exchanging $x_i$ with $x_{i+1}$ and $1$ for the identity element, define a series of linear operators on $\mathbb{C}[x_1,\cdots x_n]$ by: 
$$
\begin{aligned}
\partial_i & :=\frac{1-s_i}{x_i-x_{i+1}}, \\
\pi_i & :=\partial_i x_i\text{, and} \\
\theta_i & :=x_{i+1} \partial_i=\pi_i-1
.\end{aligned}
$$

Each of these operators satisfies the same commuting relations as the simple transpositions $s_i$, although while $s_i^2=1$,\begin{align*}\partial_i^2&=0,\\\pi_i^2&=\pi_i\text{, and}\\\theta_i^2&=-\theta_i.\end{align*}

If $w\in \Sn{n}$ where $w=s_{i_1}\cdots s_{i_k}$ is a reduced word, define \begin{align*}\partial_w&=\partial_{i_1}\cdots\partial_{i_k},\\
\pi_w&=\pi_{i_1}\cdots\pi_{i_k}\text{, and}\\
\theta_w&=\theta_{i_1}\cdots\theta_{i_k}.\end{align*} Then we have the following definitions:

\begin{definition}
Let $\alpha$ be a weak composition of $n$ and $\sort(\alpha)$ it's weakly decreasing rearrangement.  Let $w_\alpha$ give the shortest permutation in $\Sn{n}$ such that $w_\alpha(\alpha)=\sort(\alpha)$.  Then the \textbf{key polynomial} associated to $\alpha$ is $$\key_\alpha=\pi_{{w_\alpha}^{-1}}x^{\sort(\alpha)}.$$ 
\end{definition}
Key polynomials were first defined as characters by Demazure~ \cite{demazure1974nouvelle} and were then studied in a more combinatorial setting by Lascoux and Sch{\"u}tzenberger~\cite{lascoux1988keys} and Reiner and Shimozono~\cite{reiner1995key}. 

Demazure atoms were first defined by Demazure in \cite{demazure1974nouvelle} as well:
\begin{definition}
Let $\alpha\models_w n$ and $\sort(\alpha)$ it's weakly decreasing rearrangement.  Let $w_\alpha$ give the shortest permutation in $\Sn{n}$ such that $w_\alpha(\alpha)=\sort(\alpha)$.  Then the \textbf{Demazure atom} associated to $\alpha$ is $$\atom_\alpha=\theta_{w_\alpha^{-1}}x^{\sort(\alpha)}.$$ 
\end{definition}
It's worth noting there is some variation in the conventions for forming these bases and moreover some authors go so far as to consider any permutation in place of $w_\alpha$.  While we focus here on the bases, one could easily consider the broader family.
\begin{definition}
Let $w\in \Sn{n}$ and $w_0=[n,n-1,\cdots,1]$ be the long element.  The \textbf{Schubert polynomial} associated to $w$ is $$\mathfrak{S}_w=\partial_{w^{-1}w_0}\left({x_1^{n-1}x_2^{n-2}\cdots x_{n-1}^{1}}\right).$$ 
\end{definition}

Finally, Schubert Polynomials, defined first by Lascoux and Sch{\"u}tzenberger~\cite{alain1982polynomes}, are of particular interest in algebraic geometry since they encode Schubert classes in the cohomology of the complete flag variety, as was first observed by Fulton in \cite{fulton1992flags}.  

 Lascoux and Sch{\"u}tzenberger~\cite{lascoux1989tableaux} first showed that Schubert polynomials are a positive sum of key polynomials (or ``key positive'' as it's commonly expressed). Key polynomials are in turn atom positive, as is evident from the fact that $\pi_i=\theta_i+1$. From a geometric perspective, it is clear that the product of Schubert polynomials must be able to be written as a positive sum of Schubert polynomials, but is a well studied open problem to give a combinatorial proof.  Since Schubert polynomials are key positive, one might hope that by studying the products of key polynomials as a sum of keys, one might be able to indirectly study the product of Schubert polynomials: one impediment to this approach is that the product of key polynomials is not always a positive sum of key polynomials.  In contrast, it was first conjectured by Victor Reiner and Mark Shimozono (in an unpublished manuscript) that the product of key polynomials is always a positive sum of Demazure atoms.  Since key polynomials are a positive sum of Demazure atoms, fully understanding the product of Demazure atoms may help determine the product of two Schubert polynomials. 

\subsection{Symmetric and Quasisymmetric operators}
While Schur functions were originally defined as a quotient, they can also be defined using divided difference operators.  In particular, for $\lambda\vdash n$, one may define $$s_\lambda(x_1,x_2,\cdots,x_n)=\pi_{w_0}x^\lambda.$$  From this definition, it's easy to see that the Schur functions are a special case of the key polynomials when $\alpha$ is the reverse of $\lambda$ (with $\lambda$ written with sufficient zeroes at the the end to be of length $n$.)

Hivert~\cite{Hiv2000} defined a modified ``quasisymmetric action'' of $\Sn{n}$ on weak compositions (and by extension on monomials) by 
\[\qs_i(\alpha) = \left\{\begin{array}{ll} (\alpha_1,\ldots, \alpha_{i+1},\alpha_i,\ldots)& \text{ if }\alpha_i=0 \text{ or }\alpha_{i+1}=0\\
\alpha & \text{ otherwise}
\end{array}\right. .\]  

Going forward we use the $\tilde{\cdot}$ as above to indicate this quasisymmetric action (and differentiate it from the classical ``symmetric'' action that always exchanges positions $i$ and $i+1$).
Hivert showed that as operators, the $\qs_i$ satisfy 
\begin{align*}
\qs_i^2&=1\\
\qs_i\qs_j&=\qs_j\qs_i \quad \text{ for }|i-j|>1\\
\qs_i\qs_{i+1}\qs_i & =\qs_{i+1}\qs_i\qs_{i+1}.
\end{align*}
The action is quasisymmetric in the sense that $f\in \nonsym{n}$ is quasisymmetric iff $$\tilde{\sigma}(f)=f\text{ for all }\sigma\in \Sn{n}.$$  

 Using $\qs_i$ we can define operators 
 \begin{align*}\qpartial_i &= \frac{1}{x_i-x_{i+1}}(1-\qs_i),\\
 \qpi_i &= \qpartial_i x_i\text{, and}\\
 \qbarpi_i&=\qpi_i-1.
 \end{align*}

Given $\beta \wdash n$ and using $x^\beta=x_1^{\beta_1}\cdots x_n^{\beta_n}$, we explicitly have:
\[\qpi_i x^\beta = \left\{\begin{array}{ll} 
0 & \text{ if } \beta_{i+1}\neq 0\\
x^\beta + x^{(\beta_1, \ldots, \beta_i-1,\beta_{i+1}+1,\ldots,\beta_k)}+ x^{(\beta_1, \ldots, \beta_i-2,\beta_{i+1}+2,\ldots,\beta_k)}+\cdots + x^{s_i(\beta)}& \text{ if }\beta_{i+1}=0\\ %double check this
\end{array}\right. .\]

\begin{theorem}[Hivert~\cite{Hiv2000}]
The following relations are satisfied:
\begin{align*}
\qpi_i^2&=\qpi_i\\
\qpi_i\qpi_j&=\qpi_j\qpi_i \quad \text{ for }|i-j|>1\\
\qpi_i\qpi_{i+1}\qpi_i & =\qpi_{i+1}\qpi_i\qpi_{i+1}
\end{align*}
and \begin{align*}
\qbarpi_i^2&=-\qbarpi_i\\
\qbarpi_i\qbarpi_j&=\qbarpi_j\qbarpi_i \quad \text{ for }|i-j|>1\\
\qbarpi_i\qbarpi_{i+1}\qbarpi_i & =\qbarpi_{i+1}\qbarpi_i\qbarpi_{i+1}.
\end{align*}
\end{theorem}
It is additionally easy to see that \begin{align*}
\qpartial_i^2&=0\\
\qpartial_i\qpartial_j&=\qpartial_j\qpartial_i \quad \text{ for }|i-j|>1\\
\end{align*}

but unfortunately, these operators do not satisfy braid relations.  

Hivert noted that using these new operators and using an analogous definition to that of the Schur functions, the result was Gessel's fundamental basis for the quasisymmetric functions:
\begin{theorem}[Hivert~\cite{Hiv2000}]For $\alpha\models n$,  $$F_\alpha(x_1,x_2,\cdots,x_n)=\qpi_{w_0}x^\alpha.$$
\end{theorem}
Hivert goes on to define a quasisymmetric generalization of the Hall-Littlewood operators.  Extending this process in \cite{corteel2022compact}, Corteel, Haglund, Mandelshtam, Mason, and Williams went on to define a ``quasisymmetric" analogue of the Macdonald polynomials.

\section{Quasisymmetric operators and bases of the non-symmetric polynomials}
Going forward, we consider analogues of the above bases for $\nonsym{n}$ formed by replacing $s_i$ with $\qs_i$.  Since the $\qs_i$ can be used in this way to define Gessel's fundamental basis, we refer to these new bases as ``fundamental" analogs of the classical bases.  In this setting, it's worth first observing that there appears to be no natural analogue of the Schubert polynomials, since $\qpartial_i$ does not satisfy braid relations.  Thus we turn our attention to the key basis. 

\subsection{A fundamental key basis by another name}
In order to define the first basis, it is necessary to define the following, which gives a strong composition from a weak composition:
\begin{definition}[flat]
Let the \textbf{flat} of a weak composition $\alpha$ ($\strong(\alpha)$) be the strong composition formed by moving all 0s to the right end of $\alpha$. 
\end{definition}
This can be seen as analogous to the $\sort(\cdot)$ operator on compositions, which results in a natural partition.  Thus remembering the divided difference definition of the key polynomials suggests the following quasisymmetric analog:
\begin{definition}[Fundamental slide polynomial]\label{def:qkey}  Let $\alpha \wdash n$ and $w_\alpha(\alpha)=\strong(\alpha)$ such that $w_\alpha$ has minimal length.  Then 
\[\qslide_\alpha (x_1,\cdots x_n) = \qpi_{{w_\alpha}^{-1}}x^{\strong(\alpha)}.\]
\end{definition}
\begin{example}\begin{align*}\qslide_{(0,3,1,0,1)}&=\qpi_1\qpi_2\qpi_4\qpi_3 (x_1^3x_2^1x_3^1)\end{align*}
\end{example}
By construction, Gessel's Fundamental quasisymmetric functions are a special case of $\qslide_\alpha$.  In particular, when $\alpha$ is a strong composition of $n$ of length $k$, and $\beta$ is the weak composition of $n$ formed by adding zeros to the front of $\alpha$, we have $$F_\beta=\qslide_\alpha.$$

While Hivert does not study the resulting polynomials as a basis, he gives an explicit expansion into monomials.  To restate his formula in modern language, we need some additional definitions which relate to compositions:

\begin{definition}[cumulative sum]
Let $$\csum((\alpha_1,\ldots,\alpha_k)) =(\alpha_1,\alpha_1+\alpha_2,\ldots, \alpha_1+\cdots + \alpha_k)$$ give the \textbf{cumulative sum} of $\alpha$.
\end{definition}

\begin{definition}[dominance order]
Define a \textbf{dominance order} on weak compositions of $n$ by $\beta \doml \gamma$ if and only if for all $i$ $$\csum(\beta)_i \leq \csum(\gamma)_i.$$  \end{definition} 
\begin{definition}[refines]
For $\alpha$ and $\beta$ strong compositions, we say that $\beta$ \textbf{refines} $\alpha$, denoted $\beta \succcurlyeq \alpha$, if there exists an integer sequence $0=j_0<j_1<j_2<\cdots j_n=\ell(\alpha)$ such that $$\alpha_i = \beta_{j_{i-1}+1}+\beta_{j_{i-1}+2}+\cdots +\beta_{j_i}$$ for each $i$, that is, if each part of $\alpha$ can be obtained by summing consecutive parts of $\beta$.
\end{definition}
Thus $(1,2,1,3)\succcurlyeq (4,3)$.
It's worth noting that there is not general consensus in the literature about the direction of the inequality sign, with some authors following the same convention as here and others using the reverse.  

Hivert proves:
\begin{theorem}[\cite{Hiv2000} Theorem 3.29]\label{thm:slide} Let $\alpha \wdash n$ and $\tilde{\sigma}(\alpha)=\strong(\alpha)$ such that $\sigma$ has minimal length.  Then 
    \[\qslide_\alpha(x_1,\ldots,x_n) = \sum_{\substack{\beta\domr \alpha\\ \strong(\beta) \succcurlyeq \strong(\alpha)}} x_1^{\beta_1} \cdots x_n^{\beta_n}.\]
\end{theorem}
\begin{example}\begin{align*}\qslide_{(0,3,1,0,1)}&=x_{1}^{3} x_{2} x_{3} + x_{1}^{3} x_{2} x_{4} + x_{1}^{3} x_{3} x_{4} + x_{1}^{2} x_{2} x_{3} x_{4} + x_{1} x_{2}^{2} x_{3} x_{4} + x_{2}^{3} x_{3} x_{4} + x_{1}^{3} x_{2} x_{5} + x_{1}^{3} x_{3} x_{5}\\&\hspace{9cm} + x_{1}^{2} x_{2} x_{3} x_{5} + x_{1} x_{2}^{2} x_{3} x_{5} + x_{2}^{3} x_{3} x_{5}\end{align*}
\end{example}

Assaf and Searles~\cite{AS2017} discovered the polynomials $\qslide_\alpha$ independently via a different construction leading to the right hand side of Theorem \ref{thm:slide}, and named them the \textbf{ fundamental slide polynomials} although they appear not to have identified the connection to the divided difference operators or Hivert's work.  (We in fact use their language  and notation in this section rather than Hivert's original language.)  We will continue to use the name here, although in this context we would suggest ``fundamental key polynomial'' might be more appropriate to the naming conventions in this area.  We follow Assaf and Searles' convention and refer to these going forward as fundamental slide polynomials since they use the similar term ``quasisymmetric key polynomial'' to a different family of polynomials whose definition is motivated by their work with Konhert diagrams.

\subsection{Representation Theory of the Fundamental Slide Polynomials}
It is well known that the Schur functions are the characters of the irreducible representations of the enveloping algebra of $gl_N$, $U(gl_N)$.  Moreover, taking highest weight representatives, $\xi$, of the representation indexed by 
$\sigma(\lambda)$ it was conjectured (with a false proof) by Demazure and proven by Andersen in \cite{MR0782239} that the $$\operatorname{ch}(\mathcal{U}_0(\mathfrak{b}_+)\xi)=\key_{\sigma^{-1}(\lambda)} .$$

By connecting Hivert's work on quasisymmetric divided difference operators to Assaf and Searles work on the fundamental slide polynomials, we recognized that the fundamental slide polynomials appearing in \cite{AS2017} are in fact characters, mimicking the representation theoretic importance of the key polynomials. 

In particular, Hivert in \cite{Hiv2000} builds on the work of Krob and Thibon in \cite{NSYM4}, to show that a $\mathcal{U}_0(gl_N)$-module irreducible module $\mathbf{D}_\alpha$ identified in \cite{NSYM4} has the property that $$\mathrm{ch}(\mathbf{D}_{\alpha})=\pi_{\omega}X^{\alpha}=F_{\alpha}$$  and that for an appropriately chosen extremal vector $\xi$ and suitably generalized Borel subalgebra $\mathfrak{b}_+$, $$\operatorname{ch}(\mathcal{U}_0(\mathfrak{b}_+)\xi)=\qslide_{\sigma^{-1}(\alpha)}.$$
We omit the details here, as they are long and not as straightforward as in the classical case. Among other technicalities, $\mathcal{U}_0(gl_N)$ is only a bialgebra, and thus more complex to work with than the enveloping algebra $U(gl_N)$. The interested reader will find the details spread between \cite{Hiv2000} and \cite{NSYM4}.

\section{Fundamental Demazure Atoms}

Imitating the definition of the classical Demazure atom, we can define a fundamental analogue, which is also a basis for $\nonsym{n}$.

\begin{definition}[Fundamental Demazure atom]\label{def:qatom} 
Given a weak composition $\alpha$ and a minimal length $\sigma \in S_n$ such that $\strong(\alpha) = w_\alpha(\alpha)$, the {\em fundamental Demazure atom} is 
\[\qatom_\alpha = \qbarpi_{{w_{\alpha}}^-1} x^{\strong(\alpha)}.\]
\end{definition}

\begin{example}\begin{align*}\qatom_{(0,0,3,2,0,1)}&=\qbarpi_2\qbarpi_1\qbarpi_3\qbarpi_2\qbarpi_5\qbarpi_4\qbarpi_3 (x_1^3x_2^2x_3^1)\\&=x_{1}^{2} x_{3} x_{4}^{2} x_{6} + x_{1} x_{2} x_{3} x_{4}^{2} x_{6} + x_{2}^{2} x_{3} x_{4}^{2} x_{6} + x_{1} x_{3}^{2} x_{4}^{2} x_{6} + x_{2} x_{3}^{2} x_{4}^{2} x_{6} + x_{3}^{3} x_{4}^{2} x_{6}\end{align*}
\end{example}
In order to express the same polynomials more explicitly, it is helpful to have the following lemma.
\begin{lemma}\label{lem:pibar}
For $1\leq i\leq n-1$ and $\alpha=(\alpha_1,\ldots, \alpha_n)$ a weak composition, we have 
\[\qbarpi_ix^\alpha = \left\{\begin{array}{ll} 
0 & \text{ if } \alpha_i=\alpha_{i+1}=0 \text{ or } \alpha_i,\alpha_{i+1}>0\\
\sum_{j=1}^{\alpha_i} x^{(\alpha_1,\ldots, \alpha_i - j,j,\ldots,\alpha_n)} & \text{ if }\alpha_i>0 \text{ and }\alpha_{i+1}=0\\
 - \sum_{j=0}^{\alpha_{i+1}-1} x^{(\alpha_1,\ldots, j,\alpha_{i+1}-j,\ldots,\alpha_n)} &\text{ if }\alpha_i=0 \text{ and } \alpha_{i+1}>0.
\end{array}\right.\]
\end{lemma}
\begin{lemma}\label{lem:qatomrec}
Let $\alpha$ be a weak composition.  Then if  $\alpha_{i}= 0$ and $\alpha_{i+1}\neq 0$,
\[\qatom_{\alpha} = \qbarpi_i\qatom_{s_i(\alpha)} .\]
\end{lemma}

\begin{definition}\label{def:qshift}
The {\em fundamental shift} of a vector $v=(v_1,\ldots,v_n)$ with nonzero entries in positions $S=\{i_1<i_2<\cdots<i_k\}$ is the vector
$\qshift(v)=(w_1,\ldots,w_n)$ such that
\begin{enumerate}
\item if $j\in S$, 
\[w_j =\left\{\begin{array}{ll}1 & \text{if }j-1 \notin S\\v_j & \text{otherwise}\end{array}\right.,\]
\item if $i_j \in S$, $j<k$, and $i_j+1 \notin S$, then $w_{i_j+1}=v_{i_{j+1}}-1$, 
\item if $1 \notin S$ and $j=\min(S)$, $w_1 = v_{j}-1$, and
\item $w_j=0$ for all other $j$.
\end{enumerate}
That is, the fundamental shift can be formed by working right to left, reducing any positive entry $a$ in $v$ with a 0 beside it to 1 and shifting the remaining $a-1$ left until it sits beside the next nonzero entry.
\end{definition}
\begin{example}
    $\qshift([\fbox{0,0,3},\fbox{0,1},\fbox{4},\fbox{0,5}])=[\fbox{2,0,1},\fbox{0,1},\fbox{4},\fbox{4,1}]$
\end{example}
We use the following two technical lemmas to give an explicit characterization of $\qatom_\alpha(x_1,\cdots,x_n)$
\begin{lemma}\label{lem:siqshift1}
Let $\alpha$ be a weak composition and $i \in \{1,\ldots,\ell(\alpha)\}$ with $\alpha_i=0$.  Let $\beta$ be a weak composition such that $\alpha\doml\beta\doml\qshift(\alpha)$.  Then the composition $\gamma$ defined by $\gamma_i=\beta_i+\beta_{i+1}$, $\gamma_{i+1}=0$, and $\gamma_j=\beta_j$ for $j \notin\{i,i+1\}$ satisfies $\qs_i\alpha\doml\gamma\doml\qshift(\qs_i\alpha)$.
\end{lemma}

\begin{proof} Let $\alpha,i,\beta,$ and $\gamma$ be as above.  The statement is clearly true if $\alpha_{i+1}=0$, so assume not. Let $v=\qshift(\alpha)$ and $w=\qshift(\qs_i\alpha)$.  Denote by $S_\alpha$ the set of indices of nonzero entries in $\alpha$. Note that for all $j \notin \{i,i+1,i+2\}$, $v_j=w_j$. If $i-1 \in S_\alpha$, then $v_i=\alpha_{i+1}-1$, $w_i=\alpha_{i+1}$ and $v_{i+1}=1$.  If $i-1 \notin S_\alpha$, then $v_i = 0, w_i=1$, and $v_{i+1}=1$. Note that in both cases, $v_i+v_{i+1}=w_i$.

Thus, for $j\notin\{i,i+1\}$ we have 
\begin{align*}
\alpha_1+\cdots +\alpha_j &\leq \beta_1+\cdots+ \beta_j \\
&=\gamma_1+\cdots+\gamma_j\\
&\leq v_1+\cdots +v_j\\
&= w_1+\cdots + w_j.
\end{align*}
Since $\alpha_i=0$, 
\begin{align*}
\alpha_1+\cdots+\alpha_{i-1}+\alpha_{i+1} &=\alpha_1+\cdots +\alpha_{i+1}\\
&\leq \beta_1+\cdots +\beta_{i+1}\\
&=\gamma_1+\cdots +\gamma_i\\
&\leq v_1+\cdots + v_{i+1}\\
&=w_1+\cdots +w_i.\end{align*}
Finally, 
\begin{align*}
\alpha_1+\cdots +\alpha_{i+1}&\leq \gamma_1+\cdots +\gamma_i \\
&=\gamma_1+\cdots + \gamma_{i+1}\\
&\leq w_1+\cdots +w_i\\
&\leq w_1+\cdots + w_{i+1}.
\end{align*}
Therefore $\qs_i\alpha\doml\gamma\doml\qshift(\qs_i\alpha)$.
\end{proof}

\begin{lemma}\label{lem:siqshift2}
Let $\alpha$ be a weak composition and $i \in \{1,\ldots,\ell(\alpha)\}$ with $\alpha_i=0$.  Let $\gamma$ be a weak composition such that $\qs_i\alpha\doml\gamma\doml\qshift(\qs_i\alpha)$ and $\qbarpi_ix^\gamma \neq 0$.  Then any composition $\beta$ with $\beta_j =\gamma_j$ for $j\notin\{i,i+1\}$, $0\leq \beta_i \leq \gamma_i-1$, and $\beta_{i+1}=\gamma_i - \beta_i$ satisfies $\alpha\doml\beta\doml\qshift(\alpha)$. 
\end{lemma}
\begin{proof} Let $\alpha, i, \gamma,$ and $\beta$ be as above.  Since $\qbarpi_i x^\gamma \neq 0$ we have $\gamma_{i+1} = 0$.  Let $v=\qshift(\alpha)$ and $w=\qshift(\qs_i\alpha)$.  Then 
\begin{align*}
\alpha_1+\cdots +\alpha_{i-1}+\alpha_{i+1}&=\alpha_1+\cdots + \alpha_{i+1}\\
&\leq \gamma_1+\cdots +\gamma_i\\
&=\beta_1+\cdots + \beta_{i+1}\\
&\leq w_1+\cdots +w_i\\
&=v_1+\cdots + v_{i+1}
\end{align*}
and 
\begin{align*}
\alpha_1+\cdots + \alpha_i &\leq \gamma_1+\cdots + \gamma_{i-1}\\
&\leq \beta_1+\cdots +\beta_i\\
&\leq \gamma_1+\cdots + \gamma_i - 1\\
&\leq w_1+\cdots +w_i - 1\\
&=v_1+\cdots +v_i.
\end{align*}
Thus $\alpha \doml\beta\doml\qshift(\alpha)$.
\end{proof}

Now we are ready to give an explicit characterization of the quasisymmetric Demazure atoms:
\begin{theorem}\label{thm:qatom}
Let $\alpha$ be a weak composition.  Then 
\[\qatom_\alpha (x_i,x_2,\cdots,x_n) = \sum_{\alpha\doml\beta\doml\qshift(\alpha)} x^\beta.\]
\end{theorem}

\begin{proof}
Let $\alpha$ be a weak composition of length $n$ and $w_\alpha \in S_n$ a minimum length permutation such that $\strong(\alpha) = w_\alpha(\alpha)$.  We proceed by induction on $\ell(w_\alpha)$.

Suppose $w_\alpha = s_i$ for some $i$.  Then $\strong(\alpha)=s_i(\alpha)$, so all the zeros in $\alpha$ are at the end after a single exchange, and thus $\alpha_i = 0$, $\alpha_{i+1}\neq 0$, $\alpha_j\neq 0$ for any $1\leq j<i$, and $\alpha_{i+j}=0$ for any $0<j\leq n-i$.  Thus $$\qshift(\alpha) = (\alpha_1,\ldots,\alpha_{i-1},\alpha_{i+1}-1,1,0,\ldots).$$  Lemma~\ref{lem:pibar} shows $x^\beta$ occurs as a monomial (with nonzero coefficient) in $\qbarpi_i x^{\strong(\alpha)}$ iff $\alpha\doml\beta\doml\qshift(\alpha)$.

Assume it's true for any $\alpha$ such that $\ell(w_\alpha)<k$.  Now suppose $\ell(w_\alpha) = k$ with $\ell(w_\alpha s_i)=\ell(w_\alpha)-1$ for some $i$.  Then $\alpha_i = 0$ and $\alpha_{i+1}\neq 0$ and by Lemma~\ref{lem:qatomrec}, $\qatom_\alpha = \qbarpi_i\qatom_{\qs_i\alpha}$.  Let $\beta$ be a weak composition such that $\alpha\doml\beta\doml\qshift(\alpha)$.  Define $\gamma$ by $\gamma_j = \beta_k$ for $j\notin\{i,i+1\}$, $\gamma_i = \beta_i+\beta_{i+1}$, and $\gamma_{i+1} = 0$.  Then by Lemma~\ref{lem:siqshift1} $\qs_i\alpha\doml \gamma \doml\qshift(\qs_i\alpha)$ and $x^\gamma$ occurs as a monomial with coefficient 1 in $\qatom_{\qs_i\alpha}$ by the inductive hypothesis.  By Lemma~\ref{lem:pibar} $\qbarpi_i x^\gamma$ contains the monomial $x^\beta$ with coefficient 1. 

Conversely, let $x^\beta$ occur as monomial with nonzero coefficient in $\qatom_\alpha$.  Then $x^\beta$ occurs with nonzero coefficient in $\qbarpi_i x^\gamma$ for (one or more) $\gamma$, $\qs_i\alpha\doml\gamma\doml\qshift(\qs_i\alpha)$.  Thus, by Lemma~\ref{lem:pibar}, we have that $\beta_j = \gamma_j$ for $j \notin\{i,i+1\}$, $0\leq \beta_i\leq \gamma_i - 1$, and $\beta_{i+1} = \gamma_i - \beta_i$.  Thus $\gamma$ is uniquely determined and  Lemma~\ref{lem:siqshift2} $\alpha\doml\beta\doml\qshift(\alpha)$.  Therefore, \[\qatom_\alpha = \sum_{\alpha\doml\beta\doml\qshift(\alpha)}x^\beta.\]
\end{proof}

\begin{lemma}\label{lem:qatombasis}
The set $\{\qatom_\alpha\}$ is a basis for $\mathbb{C}[X]$.
\end{lemma}

\begin{corollary}
For every $\alpha\models_w n$, $\fparta=\qatom_\alpha$, where $\fparta$ is the fundamental particle basis introduced by Searles in \cite{Sea2020}.
\end{corollary}

In particular, Seales construction mimics the right hand side of Theorem \ref{thm:qatom}.  Again, we'd suggest that from this perspective, a more natural name would be the ``Fundamental atom polynomials'' as we chose to use here.  Since this name is not taken elsewhere in the literature, we continue to use it throughout.

\begin{theorem}
Let $\alpha$ be a strong composition of $n$.  Then 
\[F_\alpha(x_1,\ldots,x_n) = \sum_{\substack{\beta \wdash n, \ell(\beta) = n\\ \strong(\beta) = \alpha}}\qatom_\beta.\]
\end{theorem}

\begin{proof}
Recall~\cite{Hiv2000} that for $\alpha$ a strong composition of $n$, 
\begin{equation}\label{eqn:Fpibar}
F_\alpha(x_1,\ldots,x_n)=\sum_{\sigma\in S_n} \qbarpi_\sigma x^\alpha.
\end{equation}

Let $\sigma \in S_n$ and $\beta = \tilde{\sigma}(\alpha)$.  Then $\beta \wdash n$ and $\ell(\beta) = n$, tacking on trailing zeros if necessary.  Suppose $\sigma = s_{i_1}\cdots s_{i_k}$.  By Lemma~\ref{lem:pibar}, if $s_{i_j}\cdots s_{i_k}(\alpha) = s_{i_{j+1}}\cdots s_{i_k}(\alpha)$, then $\qbarpi_\sigma(\alpha) = 0$.  Thus, in~\eqref{eqn:Fpibar}, the only non-zero summands are those such that $\sigma(\alpha) = \beta$ where $\sigma$ is the minimal length permutation such that $\sigma(\alpha)=\beta$ under the quasisymmetric action.  Thus, 
\begin{align*}
F_\alpha(x_1,\ldots,x_n) & = \sum_{\sigma \in S_n} \qbarpi_\sigma x^\alpha\\
&=\sum_{\substack{\beta\wdash n, \ell(\beta) = n\\\strong(\alpha)=\beta}} \sum_{\substack{\sigma \in S_n\\ \tilde{\sigma}(\alpha) = \beta}} \qbarpi_\sigma x^\alpha \\
&=\sum_{\substack{\beta \wdash n, \ell(\beta) = n\\ \strong(\beta) = \alpha}}\qatom_\beta.
\end{align*}
\end{proof}

The familiar reader will recognize that this is analogous to how the Schur basis is expressed as a positive sum of quasisymmetric Schur basis, as determined by Haglund, Luoto, Mason, and van Willigenburg in \cite{HLMvW2011}.

In \cite{Sea2020}, Searles shows that the fundamental slide polynomials are a positive sum of the fundamental atoms:

\begin{theorem}[Searles] Let $\alpha$ be a weak composition.  Then $$\qk_{\alpha}=\sum_{\substack{\beta \succcurlyeq\alpha\\\operatorname{flat} (\alpha)=\operatorname{flat}(\beta)}} \qatom_{\beta}$$
    
\end{theorem}
The same paper gives the expansion of the classical Demazure atoms as a sum of the fundamental atoms:

\begin{theorem}[Searles]
 Let $\alpha$ be a weak composition.  Then 
$$\atom_\alpha=\sum_{T\in \HSSF(\alpha)}\qatom_{\wt(T)}$$
\end{theorem}
Here, $\HSSF(\alpha)$ is a subset of the semi-skyline augmented fillings of shape $\alpha$ first defined by Mason in \cite{Mason2008}.  We direct the reader to Searles \cite{Sea2020} for the details.

\section{Products}

\begin{definition}\label{def:shuffleset}
Let $\alpha$ and $\beta$ be weak compositions of length $n$. Let $A$ and $B$ be the words $A=(2n-1)^{\alpha_1}\cdots (3)^{\alpha_{n-1}}(1)^{\alpha_n}$ and $B=(2n)^{\beta_1}\cdots (4)^{\beta_{n-1}}(2)^{\beta_n}$.  The {\em shuffle set} of $\alpha$ and $\beta$ is 
\[\SSet(\alpha,\beta)=\{C\in A \shuffle B: \Des_A(C)\domr \alpha \text{ and }\Des_B(C)\domr\beta\}\] where $\Des_A(C)_i$ (resp. $\Des_B(C)_i$) is the number of letters from $A$ in the $i$th increasing run of $C$.
\end{definition}

In \cite{AS2017}, Assaf and Searles prove the following:
\begin{theorem}\label{thm:qktimesqk}
Let $\alpha$ and $\beta$ be weak compositions of length $n$.  Then 
\[\qslide_\alpha \cdot \qslide_\beta = \sum_{D\in \SSet(\alpha,\beta)} \qslide_{\Des(D)}.\]
\end{theorem}

We can build on this language to define the product of a fundamental slide polynomial with a fundamental atom.  As above, we'll shuffle $A$ and $B$, this time adding additional characters $*$ to create a word $D$. When computing $\Des_A(D), \Des_B(D),$ or $\Des(D)$ consider each $*$ to represent a run of size 0. 
\begin{definition}\label{def:qshuffleset}

As above, let $\alpha$ and $\beta$ be weak compositions of length $n$ and $A=(2n-1)^{\alpha_1}\cdots (3)^{\alpha_{n-1}}(1)^{\alpha_n}$ and $B=(2n)^{\beta_1}\cdots (4)^{\beta_{n-1}}(2)^{\beta_n}$.  A word $D$ is in the {\em fundamental shuffle set} of $\alpha$ and $\beta$ ($D\in \qSSet(\alpha,\beta)$) if:
\begin{itemize}
\item $D\in A\shuffle B\shuffle \underbrace{***\dots*}_k$ with $k$ chosen so that there are a total of $n$ runs in $D$.
\item If $D$ contains the subsequence $c**\cdots** d$ where $c,d\neq *$, then $c>d$.  That is, descents surround any subsequence of $*$'s (except possibly at the beginning or the end of $D$).
\item $\strong(\Des_A(D))\succcurlyeq\strong(\alpha)$ and $\Des_A(D)\domr \alpha$.
\item $\qshift(\beta)\domr \Des_B(D)\domr \beta$
\end{itemize}
\end{definition} 
\begin{example}\label{ex:FtimesA}  Adding dividers to mark the descents, we have
\begin{align*}&\qSSet((0,2,1),(0,3,1))\\&\hspace{.3cm}=\left\{ *|33444|12,4|3344|12,33|1444|2,34|344|12,44|334|12,344|34|12,334|144|2,3344|14|2\right\}\end{align*}
\end{example}

\begin{theorem}\label{thm:qktimesqda}
Let $\alpha$ and $\beta$ be weak compositions of length $n$.  Then 
\[\qslide_\alpha \cdot \qatom_\beta = \sum_{D\in \qSSet(\alpha,\beta)} \qatom_{\Des(D)}.\]
\end{theorem}

\begin{proof}
Let $\alpha$ and $\beta$ have length $n$.  Then 
\begin{equation}\label{eq:RHS}
\sum_{D\in \qSSet(\alpha,\beta)} \qatom_{\Des(D)} = \sum_{D \in \qSSet(\alpha,\beta)}\sum_{\Des(D)\doml \gamma \doml\qshift(\Des(D))} x^\gamma.
\end{equation}
On the other hand, 
\begin{equation}\label{eq:LHS}
\qslide_\alpha\cdot \qatom_\beta = \left(\sum_{\substack{\nu\domr\alpha\\\strong(\nu)\succcurlyeq\strong(\alpha)}}x^\nu\right)\left(\sum_{\beta\doml \delta \doml \qshift(\beta)}x^\delta\right).
\end{equation}
 Let $X=\{(\nu,\delta): \nu \domr\alpha, \strong(\nu)\succcurlyeq \strong(\alpha), \beta \doml \delta \doml\qshift(\beta)\} $ and $Y=\{(D,\gamma): D\in \qSSet(\alpha,\beta), \Des(D)\doml \gamma \doml \qshift(\Des(D))\}$.

Define $\psi:X\rightarrow Y$.  Let $(\nu,\delta) \in X$ with $\ell(\nu)=\ell(\delta)=k$.  Let $\gamma=\nu+\delta$.  Let $A=(2n-1)^{\alpha_1}\cdots (3)^{\alpha_{n-1}}(1)^{\alpha_n}$ and $B=(2n)^{\beta_1}\cdots (4)^{\beta_{n-1}}(2)^{\beta_n}$.  Create $A'$ from $A$ by placing $k-1$ vertical dividers in each word such that in $A'$ there are $\nu_1$ increasing entries before the first divider, $\nu_k$ increasing entries after the last divider, and $\nu_i$ increasing entries between the $i-1$st and $i$th dividers for $1<i<k$.  Create $B'$ from $B$ similarly, using $\delta$.  Create a new word $C$ from $A'$ and $B'$ by placing the elements in the $i$th part of $A'$ and the $i$th part of $B'$ together in the $i$th part of $C$ and rearranging as needed so the elements are in increasing order.  In particular, the restrictions on $\nu$ and $\delta$ force that there can only be one distinct integer of each parity in each part.  Leaving the dividers in $C$, note that the vector of sizes of the parts of $C$ is $\gamma$ by construction. 

Next, our goal is to adjust $C$ so that the resulting word $D$ has descents dividing the nonzero parts.  Consider any two consecutive nonempty parts $i$ and $j$ in $C$, (i.e. jumping over any parts that are empty).  If the last entry in part $i$ is not larger than the first entry in $j$, move the entries in $i$ to part $j$ to begin forming a new word $D$, in which the $i$th part empty and the $j$th part strictly increasing.  Thus all parts will be divided by a descent, after deleting the empty parts.  Replace any empty parts with a $*$.  Remove all remaining dividers to get $D$.  Note that since we move integers only into the nearest parts that already contain integers, $\Des(D)\doml \gamma \doml\qshift(\Des(D))$.

Then $\psi((\nu,\delta)) = (D,\gamma)$. 

We must show show that the image of $\psi$ is contained in $Y$.
 Since $\nu \domr\alpha, \strong(\nu)\succcurlyeq \strong(\alpha), \beta \doml \delta \doml\qshift(\beta)$, if we split $C$ according to $\gamma$, the resulting composition of odd parts is $\nu$, while the composition of odd parts is $\delta$.  When we move integers in $C$ to their position in $D$, if $i$ and $j$ contain parts of the same parity, it must be that they are exactly the same integer (or else we would have a descent already between the $i$th and $j$th part).  Thus $\strong(\Des_A(D))\succcurlyeq\strong(\alpha)$ and $\Des_A(D)\domr \alpha$ and
 $\qshift(\beta)\domr \Des_B(D)\domr \beta$ and since $C$ was a shuffle of $A$ and $B$, the elements of $A$ and $B$ remain in order in $D$.  

Define the inverse $\phi: Y\rightarrow X$ by the following process:  Let $(D,\gamma) \in Y$ with $\ell(\gamma)=k$.  Place $k-1$ vertical dividers in $D$ between entries such that the number of entries in $D$ between the $i-1$st and $i$th dividers is $\gamma_i$ for $1<i<k$, the number of entries prior to the first divider is $\gamma_1$, and the number after the $k-1$st divider is $\gamma_k$.  The number of even values before the first divider is $\delta_1$ and the number of odd values before the first divider is $\nu_1$.  Similarly, the number of even values after the $k-1$st divider is $\delta_k$ and the number of odd values is $\nu_k$. Then $\phi((D,\gamma)) = (\nu, \delta)$. 
\end{proof}

\begin{example}
Let $\alpha=(2,0,3,0,1)$ and $\beta=(0,0,2,0,2)$.  We show an example of $\psi$.  For $(\nu,\delta) \in Y$ where $\nu=(2,2,1,0,1)$ and $\delta=(0,0,2,1,1)$.  Then $A=995551$ and $B=6622$ so $A'=99|55|5||1$ and $B'=||66|2|2$.  Thus $C=99|55|566|2|12$ and $\gamma=(2,2,3,1,2)$.  Since the second $|$ in $C$ does not occur at a descent, $D=99||55566|2|12 = 99*55566212$ and $\Des(D) = (2,0,5,1,2)$.      
\end{example}
\begin{example} 
 Continuing Example \ref{ex:FtimesA}
 \begin{align*}\qslide_{(0,2,1)}\cdot \qatom_{(0,3,1)}&= \qatom_{(0,5,2)}+\qatom_{(1,4,2)}+\qatom_{(2,4,1)}+2\qatom_{((2,3,2)}+\qatom_{(3,2,2)}+\qatom_{(3,3,1)}+\qatom_{(4,2,1)}
\end{align*}
\end{example}
\begin{remark} Searles in \cite{Sea2020} gives the product of $s_{\lambda}(x_1,x_2,\cdots,x_n)\qatom_{\beta}$ as a sum of fundamental atoms and refers to it as a Littlewood-Richardson formula.  Since $\qslide_\alpha$ can be specialized to give Gessel's fundamental basis, this is a more refined quasisymmetric analogue of the same when $\alpha$ is a strong composition.
    
\end{remark}

\subsection{Structure constants for the Fundamental Atoms}
Assaf and Searles in \cite{AS2017} give the structure constants for the fundamental slide polynomials, which are positive and given above.  Here we give a formula for the product of two fundamental atoms.  

Like in the classical case, the product of two fundamental atoms is not  fundamental atom positive, although it's a well studied problem to look at particular cases when the product is positive. (See \cite{pun2016decomposition} and \cite{MR2737282}, discussed further below.)

Let $\alpha\models_w n$ and $\beta\models_w m$ be weak compositions of $n$ of length $k$.  The product of two fundamental atoms is not fundamental atom positive but each coefficient is always integral.   In this section we define a set $\qamp{\alpha}{\beta}$ of ordered multiset partitions such that $$\qatom_\alpha\cdot \qatom_\beta=\sum_{S\in\qamp{\alpha}{\beta}}(-1)^{\sgn{S}}\qatom_{\wt(S)},$$ where this sum is cancellation free.

The relevant ordered multisets partition a multiset of positive integer elements, some of which are barred or circled.  We refer to the underlying integer as the  ``value'' of the element.  To define the set, we need the following conventions:
$$\circled{1}<\overline{\circled{1}}<1<\overline{1}<\circled{2}<\overline{\circled{2}}<2<\overline{2}\cdots$$

Let $\gamma\in \{\alpha,\beta\}$. If $\gamma_i\neq 0$ and there exists a nonzero to the left of $\gamma_i$ in $\gamma$, let $j(\gamma,i)-1$ be the first nonzero entry found when starting from the immediate left of $\gamma_i$ and scanning further left. If no such nonzero entry exists, $j(\gamma,i)=1$.  (That is, $j(\gamma,i)$ is the position where $\gamma_i-1$ is placed in the fundamental shift of $\gamma$.)

\begin{definition}[fundamental atom multiset partition]  $S=(S_1,S_2,\cdots,S_k)$ is a \textit{fundamental atom multiset partition of type $\alpha$, $\beta$} ($S\in \qamp{\alpha}{\beta}$) if:

\begin{itemize}
\item $S$ partitions the multiset $$\bigcup_{i:\alpha_i\neq 0}\{{\overline{i}}^{(\alpha_i-1)},\overline{\circled{i}}\}\cup \bigcup_{i:\beta_i\neq 0}\{{i}^{(\beta_i-1)},\circled{i}\},$$
where $s^{(m)}$ indicates $s$ occurs with multiplicity $m$ in the multiset.  That is, there are $\alpha_i$ $\overline{i}$'s, one of which is circled, and $\beta_i$ $i$'s, one of which is circled.
\item  $\{\overline{i}, \overline{\circled{i}}\}$ can only occur with positive multiplicity in $S_{j(\alpha,i)},S_{j(\alpha,i)+1}, \cdots, S_i$ and  $\{{i}, {\circled{i}}\}$ can only occur with positive multiplicity in $S_{j(\beta,i)},S_{j(\beta,i)+1}, \cdots, S_i$.  That is, terms can only occur as far right in the multiset as their value, and can only move 
left in the multiset if they move left into positions corresponding to the  zeroes in $\alpha$ (or $\beta$) immediately to the left of $\alpha_i$ (or $\beta_i$ respectively).
\item If only one of $\overline{\circled{i}}$ or $\circled{i}$ exist, it must be in $S_i$. If both exist, $\overline{\circled{i}}$ must be in $S_i$ and if $\circled{i}$ is in $S_j$ for $j<i$, every other element with value $i$ must be in $S_k$ for some $k\leq j$.
\item If the elements of each set $S_i$ are listed in weakly decreasing order, then there must be an ascent between the end of one set and the start of the next nonempty set.
\end{itemize}
For such an $S$, let $\sgn{S}$ give the number of $i$ such that $\circled{i}$ is in a set left of $\overline{\circled{i}}$ in $S$ and let $\wt(S)=[|S_1|,|S_2|,\cdots, |S_k|]$
\end{definition}
\begin{example}  We omit the set notation for each $S_i$ and give each set in weakly decreasing order, so that the ascents between sets are obvious.
       \begin{align*}\qatom_{003}\cdot \qatom_{003}&=(\emptyset,\emptyset,\overline{3}\overline{3}33\overline{\circled{3}}\circled{3})+(\emptyset,\tikzmarknode{1a}{3},\tikzmarknode{1b}{\overline{3}}\overline{3}3\overline{\circled{3}}\circled{3})+(\emptyset,\tikzmarknode{2a}{3}\tikzmarknode{2c}{3},\tikzmarknode{2b}{\overline{3}}\tikzmarknode{2d}{\overline{3}}\overline{\circled{3}}\circled{3})+(\emptyset,\tikzmarknode{3c}{\overline{3}}\tikzmarknode{3a}{3},\tikzmarknode{3b}{\overline{3}}3\overline{\circled{3}}\circled{3})\\&+(\emptyset,\tikzmarknode{4c}{\overline{3}}\tikzmarknode{4a}{3}\tikzmarknode{4d}3,\tikzmarknode{4b}{\overline{3}}\overline{\circled{3}}\circled{3})
    -(\emptyset, {\overline{3}{\overline{3}}3\tikzmarknode{5a}{3}}\circled{\tikzmarknode{5b}{{3}}},\overline{\circled{\tikzmarknode{5c}{3}}})+(\tikzmarknode{6a}{3},\emptyset,\tikzmarknode{6b}{\overline{3}}\overline{3}3\overline{\circled{3}}\circled{3})
    +(\tikzmarknode{7d}{3},\tikzmarknode{7c}{\overline{3}}\tikzmarknode{7a}{3},\tikzmarknode{7b}{\overline{3}}\overline{\circled{3}}\circled{3})\\& -(\tikzmarknode{8d}{3}, {\tikzmarknode{8e}{\overline{3}}{\overline{3}}\tikzmarknode{8a}{3}}\circled{\tikzmarknode{8b}{{3}}},\overline{\circled{\tikzmarknode{8c}{3}}})+
    (\tikzmarknode{9c}{\overline{3}}\tikzmarknode{9a}{3},\emptyset,\tikzmarknode{9b}{\overline{3}}3\overline{\circled{3}}\circled{3})+ (\tikzmarknode{10a}{3}\tikzmarknode{10c}{3},\emptyset,\tikzmarknode{10b}{\overline{3}}\overline{3}\overline{\circled{3}}\circled{3})
-(\tikzmarknode{11d}{3}\tikzmarknode{11f}{3}, {\tikzmarknode{11e}{\overline{3}}{\overline{3}}}\circled{\tikzmarknode{11b}{{3}}},\overline{\circled{\tikzmarknode{11c}{3}}})
\\&-(\tikzmarknode{12f}{\overline{3}}\tikzmarknode{12d}{3}, {{\overline{\tikzmarknode{12g}{3}}}\tikzmarknode{12e}{3}}\circled{\tikzmarknode{12b}{{3}}},\overline{\circled{\tikzmarknode{12c}{3}}})
+(\tikzmarknode{13c}{\overline{3}}\tikzmarknode{13a}{3}\tikzmarknode{13d}3,\emptyset,\tikzmarknode{13b}{\overline{3}}\overline{\circled{3}}\circled{3})
-(\tikzmarknode{14c}{\overline{3}}\tikzmarknode{14a}{3}\tikzmarknode{14d}3,\tikzmarknode{14b}{\overline{3}}\circled{\tikzmarknode{14e}{3}},\overline{\circled{\tikzmarknode{14f}{3}}})
 -( {\overline{3}{\overline{3}}3\tikzmarknode{15a}{3}}\circled{\tikzmarknode{15b}{{3}}},\emptyset,\overline{\circled{\tikzmarknode{15c}{3}}})
   \\&=\qatom_{006} + \qatom_{015} + 2  \qatom_{024} + \qatom_{033} - \qatom_{051} + \qatom_{105} + \qatom_{123} - \qatom_{141} + 2  \qatom_{204} - 2  \qatom_{231} \\&\hspace{1in}+\qatom_{303} - \qatom_{321} - \qatom_{501}
   \end{align*}
  \begin{align*}\qatom_{102}\cdot \qatom_{002}&=(\overline{1},\emptyset,\overline{3}3 \overline{\circled{3}}\circled{3})+(\overline{1},\tikzmarknode{01a}{3}, \overline{\tikzmarknode{01b}{3}}\overline{\circled{3}}\circled{3})+(\tikzmarknode{02b}{3}\overline{\tikzmarknode{02a}{1}},\emptyset, \overline{{3}}\overline{\circled{3}}\circled{3})
     -(\overline{1}, {\overline{3}\tikzmarknode{03a}{3}}\circled{\tikzmarknode{03c}{3}},{\overline{\circled{\tikzmarknode{03d}{3}}}})
     \\
     &\hspace{1in}-(\tikzmarknode{04f}{3}\overline{\tikzmarknode{04e}{1}},\overline{\tikzmarknode{04a}{3}}\circled{\tikzmarknode{04c}{3}},{\overline{\circled{\tikzmarknode{04d}{3}}}})
     \\&=\qatom_{104} + \qatom_{113}+ \qatom_{203} - \qatom_{131}  - \qatom_{221}
     \end{align*}

\end{example}

\begin{theorem}Let $\alpha$ and $\beta$ be weak compositions of length $n$. Then  $$\qatom_\alpha\cdot \qatom_\beta=\sum_{S\in\qamp{\alpha}{\beta}}(-1)^{\sgn{S}}\qatom_{\wt(S)}.$$
\end{theorem}
\begin{proof}
Recall the expansion of the fundamental atoms into monomials in Theorem \ref{thm:qatom}.  Every term in the monomial expansion of $\qatom_\alpha\cdot \qatom_\beta$ can be placed in bijection with pairs of multiset partitions $(T,U)=\left((T_1,\cdots, T_k),(U_1,\cdots,U_k)\right)$  such that $T$ and $U$ respectively partition $$\bigcup_{i:\alpha_i\neq 0}\{{\overline{i}}^{(\alpha_i-1)},\overline{\circled{i}}\}\text{ and }\bigcup_{i:\beta_i\neq 0}\{{i}^{(\beta_i-1)},\circled{i}\},$$ all circled elements with value $i$ are in $T_i$ or $U_i$, and elements  $\{\overline{i}, \overline{\circled{i}}\}$ can only occur with positive multiplicity in $T_{j(\alpha,i)},T_{j(\alpha,i)+1}, \cdots, T_i$ and  $\{{i}, {\circled{i}}\}$ can only occur with positive multiplicity in $U_{j(\beta,i)},U_{j(\beta,i)+1}, \cdots, U_i$. Then summing over all such pairs $(T,U)$, we have that the left hand side is $$\sum_{(T,U)}x_1^{|T_1|+|U_1|}x_2^{|T_2|+|U_2|} \cdots x_k^{|T_k|+|U_k|}.$$  Let $V(T,U)=(V_1,V_2,\cdots,V_K)=(T_1\cup U_1,T_2\cup U_2,\cdots, T_k\cup U_k)$

To expand the right hand side in terms of monomials, start with any element $S=(S_1,\cdots,S_K)$ in $\qamp{\alpha}{\beta}$.  We create a set of new ordered multiset partition $ W(S)$ if for any $i$ such that $|S_i|>1$  we distributing all the elements of $S_i$ weakly left into   $S_{j(\wt(S),i)}\cup S_{j(\wt(S),i)+1}\cup\cdots \cup S_{i}$.  To ensure that $W(S)$ generates the monomial term expansion of $\qatom_{\wt{(S)}}$, we must fix an order to move elements to the left:  In particular, distribute elements of $S_i$ leftward in all possible ways so that \begin{enumerate}
    \item their relative order remains weakly decreasing (when we again sort elements of each particular set in decreasing order) to create a series of ordered multiset partitions and
    \item at least one element remains in $S_i$.  
\end{enumerate}  

After moving the elements left in all possible ways while preserving this order, the result is $W(S)$ and $$\qatom_{\wt{(S)}}=\sum_{W\in W(S)}x_1^{|W_1|}\cdots x_k^{|W_k|},$$ with the sign of $W$ being assigned as in $\qamp_{\alpha}{\beta}$.  Note that by construction, $\sgn{S}= \sgn{W}$ for all $W\in W(S)$.  Also note that one can reconstruct $S$ from any element $W\in W(S)$: in particular we move the elements of $W_i$ into the next nonempty set  $W_{j}$ exactly when all of the elements in $W_i$ are larger than all of the elements of $W_{j}$ .

Next, we perform a sign reversing involution on the set $$X=\cup_{S\in \qamp{\alpha}{\beta}}W(S).$$  In particular, for every $W\in X$ such that there exists an  $\overline{\circled{i}}$ and $\circled{i}$ in different sets, pick the smallest such $i$ and reverse their position.  By construction, they must always occur at the end of one set and the start of the next, so this switch can always be done without affecting the relative order of the elements within each set.  In particular, this means that when $\overline{\circled{i}}$ is left of $\circled{i}$, all of the elements in the set containing $\overline{\circled{i}}$ are larger than the elements in the set containing  $\circled{i}$ as desired for elements of $X$ when there is a weak descent between sets.  Moreover, the only fixed points are elements where any existing pairs $\overline{\circled{i}}$ and $\circled{i}$ are in the same set $W_i$.  But this is exactly the elements in the sum over monomials of the left hand side, as desired.    

\end{proof}

As is the case with the classical Demazure atoms, there are special cases where the above product is positive:
\begin{corollary}
    Let $\alpha$ be a (strong) composition of length $k$ and $\beta$ be a weak composition of length $n\geq k$.  Form $\alpha'$ from $\alpha$ by adding zeros to the end so that $\alpha'$ and $\beta$ have the same length. Then  $\qatom_\alpha\cdot \qatom_\beta$
    is a positive sum of fundamental atoms.
\end{corollary}
In particular, the analogous statement for the product of Demazure atoms, with one term indexed by partitions was proven by Pun in her PhD. thesis \cite{pun2016decomposition}, building on previous work by Haglund, Luoto, Mason, and van Willigenburg in \cite{MR2737282}: 

\begin{corollary}
    Let $\alpha$ and $\beta$ be a weak composition of length $k$. Assume for all $i$ if $\alpha_i \cdot \beta_i\neq 0$, $ $ Then  $\atom_\alpha\cdot \atom_\beta$
    is a positive sum of fundamental atoms.
\end{corollary}

\begin{acknowledgement}
 This work was inspired by a collaboration started by both authors at Banff International Research Center and continued at the Simons Laufer Mathematical Sciences Institute. The first author also wishes to gratefully acknowledge the Institut Mittag-Leffler for graciously hosting her during a portion of this work.
\end{acknowledgement}

\printbibliography

\end{document}